%% file: main.tex
\documentclass{article}

\usepackage{microtype}
\usepackage{graphicx}
\usepackage{subfigure}
\usepackage{booktabs} 

\usepackage{hyperref}
\usepackage{xspace}

\usepackage{tikz}
\usetikzlibrary{positioning,calc,fit}
\usepackage{fancyhdr}
\usepackage{amsmath,amsthm}
\usepackage{amssymb}
\usepackage{mathtools}
\usepackage{thmtools}
\usepackage{thm-restate}
\RequirePackage{latexsym}
\usepackage{nicefrac}
\RequirePackage{dsfont}
\usepackage{stmaryrd,scalerel}
\usepackage{cleveref}
\RequirePackage{bm}
\RequirePackage{xspace}
\RequirePackage{mathrsfs}
\usepackage{multibib}
\usepackage{tikz-cd}
\usepackage{adjustbox}
\newcommand{\Tb}{\mathbf{T}}

\DeclareMathOperator\erf{erf}
\newcommand{\veps}{\varepsilon}

\newcommand{\Xsf}{\mathsf{X}}

\newcommand{\Zsf}{\mathsf{Z}}
\newcommand{\Hc}{\mathcal{H}}
\newcommand{\Lc}{\mathcal{L}}

\newcommand{\Bc}{\mathcal{B}}
\newcommand{\Nc}{\mathcal{N}}
\newcommand{\RR}{\mathds{R}}
\newcommand{\EE}{\mathbf{E}}
\newcommand{\Tpush}{T_{\#}}

\newcommand{\Tbpush}{\Tb_{\#}}
\newcommand{\xv}{\mathbf{x}}
\newcommand{\Uv}{\mathbf{U}}

\newcommand{\Iv}{\mathbf{I}}

\newcommand{\Av}{\mathbf{A}}
\newcommand{\Zv}{\mathbf{Z}}

\newcommand{\rank}{\mathrm{rank}}
\newcommand{\muv}{\boldsymbol{\mu}}
\newcommand{\zv}{\mathbf{z}}
\newcommand{\rv}{\mathbf{r}}
\newcommand{\av}{\mathbf{a}}
\newcommand{\vv}{\mathbf{v}}

\newcommand{\wv}{\mathbf{w}}
\newcommand{\uv}{\mathbf{u}}
\newcommand{\bv}{\mathbf{b}}
\newcommand{\grad}{{\nabla}}
\newcommand{\eg}{{e.g.}\xspace}
\newcommand{\ie}{{i.e.}\xspace}
\newcommand{\iid}{{i.i.d.}\xspace}

\newcommand{\etc}{{etc.}\xspace}
\newcommand*\rmd{\mathop{}\!\mathrm{d}}

\newtheorem{theorem}{Theorem}
\newtheorem{lemma}{Lemma}
\newtheorem{remark}{Remark}
\newtheorem{corollary}{Corollary}
\newtheorem{definition}{Definition}
\newtheorem{example}{Example}
\newtheorem{proposition}{Proposition}
\usepackage[accepted]{icml2020}

\icmltitlerunning{Tails of Lipschitz Triangular Flows}

\begin{document}

\twocolumn[
\icmltitle{Tails of Lipschitz Triangular Flows}



\icmlsetsymbol{equal}{*}

\begin{icmlauthorlist}
\icmlauthor{Priyank Jaini}{equal,uw,vec}
\icmlauthor{Ivan Kobyzev}{bor}
\icmlauthor{Yaoliang Yu}{uw,vec}
\icmlauthor{Marcus A. Brubaker}{bor,york}
\end{icmlauthorlist}

\icmlaffiliation{uw}{Univesity of Waterloo, Waterloo, Canada}
\icmlaffiliation{vec}{Vector Institute, Toronto, Canada}
\icmlaffiliation{york}{York University, Toronto, Canada}
\icmlaffiliation{bor}{Borealis AI}

\icmlcorrespondingauthor{Priyank Jaini}{p.jaini@uwaterloo.ca}

\icmlkeywords{Machine Learning, ICML}

\vskip 0.3in]



\printAffiliationsAndNotice{\icmlEqualContribution} 

\begin{abstract}
 We investigate the ability of popular flow based methods to capture tail-properties of a target density by studying the increasing triangular maps used in these flow methods acting on a tractable source density. We show that the density quantile functions of the source and target density provide a precise characterization of the slope of transformation required to capture tails in a target density. We further show that any Lipschitz-continuous transport map acting on a source density will result in a density with similar tail properties as the source, highlighting the trade-off between a complex source density and a sufficiently expressive transformation to capture desirable properties of a target density. Subsequently, we illustrate that flow models like Real-NVP, MAF, and Glow as implemented originally lack the ability to capture a distribution with non-Gaussian tails. We circumvent this problem by proposing tail-adaptive flows consisting of a source distribution that can be learned simultaneously with the triangular map to capture tail-properties of a target density. We perform several synthetic and real-world experiments to compliment our theoretical findings. 
\end{abstract}

\input{section/intro.tex}

\input{section/prelim.tex}

\input{section/uni.tex}

\input{section/mul.tex}

\input{section/flow.tex}
\input{section/con.tex}

\bibliography{nft.bib}
\bibliographystyle{icml2020}

\appendix
\onecolumn
\section{Proofs}
\label{sec:proofs}

\qslope*
\begin{proof}
We prove this by contradiction. Assume on the contrary that there exists a diffeomorphism $T: \mathbb{R} \to \mathbb{R}$, such that $q = T_{\#}p$ and $\exists~ M >0,~ z_0>0 $, such that $\forall z > z_0, \ T'(z) \le M$. Because $T$ is a univariate diffeomorphism, it is a strictly monotonic function. Without loss of generality, consider a strictly increasing function $T$, such that $0 < T'(z) \le M$ for all $z>z_0$.  Since, $p \in \Lc$, we have
\begin{align}
 \int_{\Zsf} e^{\lambda_1 z}p(z) \rmd z < \infty, \qquad \text{for some } \lambda_1 >0 
\end{align}
Furthermore, since $q \in \Hc$, we have
\begin{align}
    \int_{\Xsf}e^{\lambda x}q(x) \rmd x &= \infty, \qquad \forall~ \lambda >0 \\
    \implies \qquad \int_{\Zsf}e^{\lambda T(z)}p(z) \rmd z &= \infty, \qquad \forall \lambda >0, \qquad [\because \text{change of variables}]  
 \end{align}   
 Split the domain $\Zsf$ into : $\Zsf_+ = \Zsf \cap \{z \ge 0 \}$ and $\Zsf_- = \Zsf \cap \{z < 0 \} $. The integral over the negative part trivially converges since:
 \[
 \int_{\Zsf_-}e^{\lambda T(z)}p(z) \rmd z  \le \int_{\Zsf_-}e^{\lambda T(0)}p(z) \rmd z \le e^{\lambda T(0)},
 \]
 where we used that $T$ is increasing. Next, we split the integral over $\Zsf_+ $ into two parts: integral from $0$ to $z_0$ and from $z_0$ to $\infty$. The first integral is clearly finite since it is an integral of a continuous function over a compact set in $\RR$. Thereafter, integrating the inequality on a slope, we get $ \forall ~ z > z_0$: $T(z) \le Mz + T(z_0) $.  Then:
\begin{align}   
      \int_{z_0}^\infty e^{\lambda T(z)}p(z) \rmd z &\leq e^{T(z_0)}\int_{z_0}^\infty e^{\lambda Mz}p(z) \rmd z.\qquad \forall~\lambda >0 \\
      & \le e^{T(z_0)}\int_{  \Zsf}e^{\lambda Mz}p(z) \rmd z, \qquad \forall~ \lambda >0
\end{align}      
Choose $\lambda$ such that $\lambda M = \lambda_1$. Then, the integral must be  finite because $p$ is light-tailed leading to the desired contradiction.
\end{proof}

\bsup*
\begin{proof}
Let $0 < \alpha < 1$. 
\begin{align*}
    fQ_p(u) \sim (1-u)^{\alpha} &\iff Q(u) \sim (1-u)^{\delta} + c, \qquad 0< \delta < 1, ~ c \text{ is a finite constant} \\
    &\iff \lim_{u \to 1^-}Q(u) \to c \\
    &\iff F_p^{-1}(1) = c \iff p \text{ has support bounded from above. }
\end{align*}
A similar argument proves the reverse direction.
\end{proof}

\exp*
\begin{proof}
\begin{align}
    \int_{z_0}^\infty z^\omega p(z) dz  \text{ exists } &\iff \int_{u_0}^1 Q^{\omega}_p(u)\rmd u \text{ exists for some $u_0>0$}  \\
    &\iff \int_{u_0}^{1-\epsilon} Q^{\omega}_p(u)\rmd u\text{ exists} \quad \& \quad \int_{1-\epsilon}^{1} Q^{\omega}_p(u)\rmd u\text{ exists}   
\end{align}
The first integral is finite because the integrand is non-singular. For the second integrand, we can use the asymptotic behaviour of the quantile function by choosing $\epsilon$ very close to $1$. Subsequently, the integral exists and converges if and only if $1 -\omega \gamma > 0 \iff \omega < \frac{1}{\gamma} $.
\end{proof}

\rate*
\begin{proof}
The integral 
\begin{align}
    \mathbb{E}_q[|x|^{\omega_q - \epsilon}] & = \int_{\RR} |x|^{\omega_q - \epsilon} q(x) \rmd x \\
    &= \int_{\RR} |T(z)|^{\omega_q - \epsilon} p(z) \rmd z
\end{align}
converges for $ 0 < \epsilon < \omega_q $, because $q$ is $\omega^{-1}_q$-heavy. Because $T$ is a univariate diffeomorphism, it is a strictly monotone function. Without loss of generality, let us consider $T$ to be positive increasing function and investigate the right asymptotic. Consider the function $T(z)^{\omega_q - \epsilon}/z^{\omega_p} $ for big positive $z$. Assume there is a sequence $\{z_i \}_{i=1}^\infty$, such that $\lim_i z_i = +\infty$ and the sequence $T(z_i)^{\omega_q - \epsilon}/z_i^{\omega_p} $ does not converge to zero. In other words, there exists $a>0$, such that for any $N>0$ there exists $z_j > N$, such that $T(z_j)^{\omega_q - \epsilon}/z_j^{\omega_p} > a $. Let us work with this infinite sub-sequence $\{z_j\}$. Because $T(z) $ is increasing function, we can estimate its integral from the left by its left Riemannian sum with respect to the sequence of points $\{z_j\}$: 
$$\int_{N}^\infty T(z)^{\omega_q - \epsilon} p(z) \rmd z \ge \sum_j T(z_j)^{\omega_q - \epsilon}p(\Delta z_j) > a \sum_j z_j^{\omega_p } p(\Delta z_j). $$

Since, $p$ is $\omega^{-1}_p-$heavy, the series on the right hand side diverges as a left Riemannian sum of a divergent integral. But this contradicts to the convergence of the integral on the left hand side. Hence, our assumption was wrong and for all sequences $\{z_i\}$ we have: $\lim T(z_i)^{\omega_q - \epsilon}/z_i^{\omega_p} = 0 $. Hence,  $|T(z)|^{\omega_q - \epsilon} = o(|z|^{\omega_p})$ 
 which leads to the desired result that $|T(z)| = o(|z|^{ \nicefrac{\omega_p}{\omega_q - \epsilon}})$.
\end{proof}

\elc*
\begin{proof}
The density function of the conditional $p(x| X_1 = \xv_1) $ is proportional to 
$g_R((x - \mu^*)^T \Sigma^{*-1} (x- \mu^*))$, where $x \in \RR^{d_2}$ and $g_R$ is the same function as for the distribution of $\Xsf$ (see \cite{cambanis1981theory}). Then, because it is a $d_2$-dimensional elliptical distribution, it is $\alpha$-heavy iff   $\mu_l = \int_0^\infty r^{l+d_2-1} g_R(r^2)dr < \infty $ for all $0< l < \alpha$. It is given that $\Xsf$ is $\omega^{-1}$-heavy, which is equivalent to $\int_0^\infty r^{l+d-1} g_R(r^2)dr < \infty, \ \forall 0<l<\omega $. Because $d = d_1+d_2$, one gets that $\int_0^\infty r^{\tilde{l} + d_2 -1} g_R(r^2)dr < \infty, \ \forall~ 0<\tilde{l}<\omega + d_1 $, hence $\Xsf_2|\Xsf_1 = \xv_1$ is $(\omega + d_1)^{-1}$-heavy.
\end{proof}

\mul*
\begin{proof}
On the contrary, assume that $\Tb$ is $M-$Lipschitz. Since $q(\xv)$ is heavy tailed we have that $\forall~\lambda >0$
\begin{align}
  \int_{\xv}e^{\lambda \|\xv\|} q(\xv) \rmd \xv &= \infty \\
  \implies \int_{\zv}e^{\lambda \|T(\zv)\|} p(\zv) \rmd \zv &= \infty \\
  \int_{\zv}e^{\lambda \|T(\zv)\|} p(\zv) \rmd \zv &\leq \int_{\zv}e^{\lambda M \|\zv\|} p(\zv) \rmd \zv 
\end{align}
Since $p(\zv)$ is light-tailed there exists a $\lambda > 0$ such that the right hand side of the equation above is finite. This gives us the required contradiction. 
\end{proof}

\cormul*
\begin{proof}
We will prove this using contradiction; assume that $\forall (i,j) \in [d]^2,~ \frac{\partial T_i}{\partial z_j} \leq M < \infty$. Assume for simplicity that $\Tb(0) = c < \infty$. Therefore, we have
\begin{align}
    T_i(\zv) - T_i(0) &= \int_{\rv(0 \to \zv) : 0}^{\zv} \nabla T_i \cdot \rmd \Vec{r} \\
    \implies \qquad |T_i(\zv) - T_i(0)| &\leq M\sum_{i=1}^d|z_i| 
\end{align}
Since, $q(\zv)$ is heavy tailed, $\exists~ \uv \in \Bc_1$ such that $\forall~ \kappa >0$
\begin{align}
    \int_{\RR^d} e^{\kappa \uv^T \xv} q(\xv) \rmd \xv&= \infty \\
    \ie \qquad \int_{\RR^d} e^{\kappa \uv^T \Tb(\zv)} p(\zv) \rmd \zv &= \infty \qquad [~\text{ change of variables}~]
\end{align}
We have
\begin{align}
    \int_{\RR^d} e^{\kappa \uv^T \Tb(\zv)} p(\zv) \rmd \zv&= \int_{\RR^d} \prod_{i=1}^d e^{\kappa u_i T_i(\zv)} p(\zv) \rmd \zv\\
    &\leq C \int_{\RR^d} \prod_{i=1}^d e^{\kappa |u_i| |T_i(\zv)|} p(\zv)\rmd \zv, \qquad [~C =\text{ finite constant}~] \\
    &\leq C \int_{\RR^d} \prod_{i=1}^d e^{\kappa M\sum_{i=1}^d  |u_i| |z_i|} p(\zv)\rmd \zv, \qquad [~u =\max |u_i|~] \\
    &\leq \tilde{C} \int_{\RR^d} e^{\kappa M \sum_{i=1}^d|u_i||z_i|} p(\zv)\rmd \zv  \\
    &= \tilde{C}\int_{\RR^d} e^{\kappa M \sum_{i=1}^d\text{sign}(z_i)|u_i|z_i} p(\zv) \rmd \zv 
\end{align}
Partition $\RR^d$ into $2^d$ sets $U_k, ~ k \in [2^d]$, \ie $\RR_d = \cup_{k=1}^{2^d} U_k$ such that if $\av = (a_1, a_2, \cdots, a_d) \in U_i$, and $\bv = (b_1, b_2, \cdots, b_d) \in U_j, ~ i \neq j$, then there exists at least one index $m \in [d]$ such that $\text{sign}(a_m) \neq \text{sign}(b_m)$. Subsequently, we can rewrite the integral above as 
\begin{align}
    \tilde{C} \int_{\RR^d} e^{\kappa M \sum_{i=1}^d\text{sign}(z_i)|u_i|z_i} p(\zv)\rmd \zv 
    &= \tilde{C} \sum_{k=1}^{2^d} \int_{U_k} e^{\kappa M \sum_{i=1}^d\text{sign}(z_i)|u_i|z_i} p(\zv) \rmd \zv \\
    &= \tilde{C} \sum_{k=1}^{2^d} \int_{U_k} e^{\kappa M \wv^T\zv} p(\zv) \rmd \zv, \qquad w_i = \text{sign}(z_i)\cdot|u_i| \\
\end{align}
We will prove that each integral over the set $U_k$ is finite.
\begin{align}
     \int_{U_k} e^{\kappa M \wv^T\zv} p(\zv) \rmd \zv~ & \leq \int_{\RR^d} e^{\kappa M \wv^T\zv} p(\zv) \rmd \zv  
\end{align}
Since $p(\zv)$ is light-tailed, we know that for any $\uv \in \Bc_1$, there exists a $\lambda > 0$ such that $\int_{\RR^d} e^{\lambda \uv^T \zv} p(\zv) \rmd \zv < \infty$. Choose any $\uv \in \Bc_1$, then for $\lambda = \kappa M / \|\wv\|$ we have that the above integral is finite. This directly implies that 
\begin{align}
    \sum_{k=1}^{2^d} \int_{U_k} e^{\kappa M \wv^T\zv} p(\zv) \rmd \zv < \infty
\end{align}
Hence, we have our contradiction.
\end{proof}

\mult*
\begin{proof}
We need to show that 
\begin{align}
    \lim_{z_j \to \infty}\frac{\partial T_{jj}}{\partial z_j} = \lim_{z_j \to \infty}\frac{fQ_{p, j| <j}}{fQ_{q, j|<j}} \to \infty, \qquad \forall~j \in [d]
\end{align}
Thus, all we need to show is that the generating variate $R^{*}$ of the conditional distribution for the target is heavier than the generating variate $S^{*}$ of the conditional distribution of the source. From \S\ref{sec:uni}, we know that the tail exponent in the asymptotics of the density quantile function characterize the degree of heaviness. Furthermore, we also know that asymptotical behaviour of the density quantile function is directly related to the asymptotical behaviour of the density function since if $f$ is a density function, the cdf is given by $F(x) = \int f(x) \rmd x$, the quantile function therefore is $Q = F^{-1}$ and the density quantile function is the reciprocal of the derivative of the quantile function \ie $fQ = \nicefrac{1}{Q'}$. Hence, we need to ensure that asymtotically, the density of $R^{*}$ is heavier than the density of $S^{*}$. Using the result of the cdf of a conditional distribution as given by Eq.(15) in \cite{cambanis1981theory} we have that asymptotically
\begin{align}
    f_{R^{*}}(x) = C x^{d_1-d}f_{R}(x)
\end{align}
where $d_1$ is the dimension of the partition that is being conditioned upon. Since, $R$ is heavier tailed than $S$, we have that $R^{*}$ is heavier tailed than $S^{*}$ for all the conditional distributions. 
\end{proof}

\Flow*
\begin{proof}
Here, we will prove the result in two-dimensions and the higher-dimensional proof will follow directly. Following the definition of class $\Hc$ and $\Lc$ as given in the beginning of \Cref{sec:uni}, we will show that for all direction vectors $\vv \in \Bc$ where $\Bc := \{\vv : \|\vv\| = 1\}$, the univariate random variable $\vv^T \xv \in \Lc$ \ie there is no direction on the hyper-sphere where the marginal distribution of the push-forward random variable is heavy-tailed.  
\begin{align*}
    \int_{\xv} \mathsf{exp}(\lambda\cdot \vv^T\xv) q(\xv) \rmd \xv &= \int_{\zv} \mathsf{exp}(\lambda\cdot \vv^T\Tb(\zv)) p(\zv) \rmd \zv \\
    &= \int_{\zv} \mathsf{exp}(\lambda v_1z_1 + \lambda v_2\cdot \sigma \cdot z_2 + \lambda v_2\cdot \mu)p(\zv) \rmd \zv \\
    &\leq \int_{\zv} \mathsf{exp}(\lambda v_1 z_1 + \lambda v_2\cdot B \cdot z_2 + \lambda v_2\cdot M \cdot z_1 )p(\zv) \rmd \zv \\
    &= \int_{\zv} \mathsf{exp}(\tilde{\lambda} \cdot \uv^T\zv)p(\zv) \rmd \zv < \infty, \qquad \forall~ \tilde{\lambda} > 0, \forall \uv \in \Bc 
\end{align*}
where $B$ is the upper bound of $\sigma(\cdot)$, $M$ is the Lipschitz constant of $\mu(\cdot)$ and the final inequality follows from the fact that $p(\zv)$ is a light-tailed distribution.
\end{proof}

\section{Useful Results, Figures, and Examples}
\label{app:ell}
\begin{example}
	\label{exm:norm2t1}
	Let $p \sim \Nc(0,1)$ and $q \sim t_1(0, 1)$. Then, $T$ such that $q:= \Tpush p$ is given by:
	\begin{align*}
	T(z) &= G^{-1} \circ F = \tan \Big(\frac{\pi}{2} \erf(\frac{z}{\sqrt{2}})\Big) \\ 
	\text{\&,} \quad T'(z) &= \sqrt{\pi}e^{-\frac{z^2}{2}}\sec^2\Big(\frac{\pi}{2} \erf(\frac{z}{\sqrt{2}})\Big)
	\end{align*}
	where $\erf(t) = \frac{2}{\sqrt{\pi}}\int_0^t e^{-s^2} \rmd s$ is the error function. 
	Furthermore, $fQ_{p}(u) \sim (1-u) \big(-2\log (1-u)\big)^{\nicefrac{1}{2}}$ and $fQ_q(u) \sim (1-u)^{2}$ and hence, $\lim_{z \to \infty}T'(z) = \lim_{u \to 1^-} (1-u)^{-1} \big(-2\log (1-u)\big)^{\nicefrac{1}{2}} \to \infty$. 

Similarly, for $p \sim \mathrm{uniform}[0,1]$:
	\begin{align*}
	T(z) &= G^{-1} \circ F = \tan \Big(\pi (z - \frac{1}{2})\Big) \\ \text{\&,} \quad T'(z) &= \pi\sec^2\Big(\pi (z - \frac{1}{2})\Big)
	\end{align*}
and $fQ_p(u) =1$. Thus, $\lim_{z \to \infty}T'(z) = \lim_{u \to 1^-} (1-u)^{-2} \to \infty$.
\end{example}

\begin{example}[Pushing uniform to normal]
	Let $p$ be uniform over $[0,1]$ and $q\sim\Nc(\mu, \sigma^2)$ be normal distributed. The unique increasing transformation 
	\begin{align*}
	T(z) &= G^{-1} \circ F = \mu + \sqrt{2}\sigma \cdot \erf^{-1}(2z-1) \\ &= \mu+\sqrt{2}\sigma\cdot\sum_{k=0}^\infty \frac{\pi^{k+1/2}c_k}{2k+1} (z-\tfrac{1}{2})^{2k+1},
	\end{align*}
	where $\erf(t) = \frac{2}{\sqrt{\pi}}\int_0^t e^{-s^2} \rmd s$ is the error function, which was Taylor expanded in the last equality. The coefficients $c_0=1$ and $c_k = \sum_{m=0}^{k-1} \frac{c_m c_{k-1-m}}{(m+1)(2m+1)}$.
	\label{exm:unif2norm}
	We observe that the derivative of $T$ is an infinite sum of squares of polynomials. Both uniform and normal distributions are considered ``light-tailed'' (all their higher moments exist and are finite). However, an increasing transformation from uniform to normal distribution has unbounded slope. Density quantile functions help us to reveal this precisely: $fQ_p(u) = 1$ and $fQ_{q}(u) \sim (1-u) \big(-2\log (1-u)\big)^{\nicefrac{1}{2}}$ \ie Normal distribution is ``relatively'' heavier tailed than uniform distribution explaining the asymptotic divergence of this transformation. However, note that this characterization does not follow immediately from \Cref{thm:qslope}. Indeed, density quantiles provide a more granular definition of heavy-tailedness based on the tail-exponent $\alpha$ and shape exponent $\beta$. 
\end{example}
\begin{lemma}[Marginal distributions of an elliptical distribution are elliptical, \cite{frahm2004generalized}]
\label{lemma:ell_mar}
Let $\Xsf = (\Xsf_1, \Xsf_2) \sim \varepsilon_d(\muv, \Sigma, F_R)$ where $\Xsf_1 \subseteq \RR^{d_1}$ and $\Xsf_2 \subseteq \RR^{d_2}$ partition $\Xsf$ such that $d_1 + d_2 = d$. Let $\muv_1 \in \RR^{d_1}, \muv_2\in \RR^{d_1}$ and $\Sigma_{11}\in \RR^{d_1 \times d_1}, \Sigma_{12} \in \RR^{d_1 \times d_2}, \Sigma_{22} \in \RR^{d_2 \times d_2}$ be the corresponding partitions of $\muv$ and $\Sigma$ respectively. Then, $\Xsf_i \sim \varepsilon_{d_i}(\muv_i, \Sigma_{ii}, F_R), ~ i \in \{1, 2\}$.
\end{lemma}

\begin{lemma}[Conditional distributions of an elliptical distribution are elliptical, \cite{cambanis1981theory, frahm2004generalized}]
\label{lemma:ell_con}
Let $\Xsf \sim \varepsilon_d(\muv, \Sigma, F_R)$ where $\muv = (\muv_1, \muv_2) \in \RR^d$ and $\Sigma \in \RR^{d \times d}$ is p.s.d with $\rank(\Sigma) = r$ and $\Sigma = \Av \Av^{T}$ where $\Xsf \stackrel{d}{=} \muv + R\Av U^{(r)}$. Further, let $\Xsf_1 \subseteq \RR^{d_1}$ and $\Xsf_2 \subseteq \RR^{d_2}$ partition $\Xsf$ such that $d_1 + d_2 = d$. Let $\muv_1 \in \RR^{d_1}, \muv_2\in \RR^{d_2}$ and $\Sigma_{11}\in \RR^{d_1 \times d_1}, \Sigma_{12} \in \RR^{d_1 \times d_2}, \Sigma_{22} \in \RR^{d_2 \times d_2}$ be the corresponding partitions of $\muv$ and $\Sigma$ respectively. If the conditional random vector $\Xsf_2 | (\Xsf_1 = \xv_1)$ exists then $$\Xsf_2 | (\Xsf_1 = \xv_1)  \stackrel{d}{=} \muv^* + R^*\Sigma^*U^{(d_2)}$$ where 
$\muv^* = \muv_2 + \Sigma_{21}\Sigma_{11}^{-1}(\xv_1 - \muv_1),~ 
    \Sigma^* = \Sigma_{22} + \Sigma_{21}\Sigma_{11}^{-1}\Sigma_{12}, ~
    R^* = \Big(\big(R^2 - h(\xv_1)\big)^{\nicefrac{1}{2}} | \Xsf_1 = \xv_1\Big)$ where $h(\xv_1) = (\xv_1 - \muv_1)\Sigma_{11}^{-1}(\xv_1 - \muv_1)^T$.
\end{lemma}

\begin{figure*}[!h]
    \centering
    \includegraphics[width=\textwidth, scale=0.5]{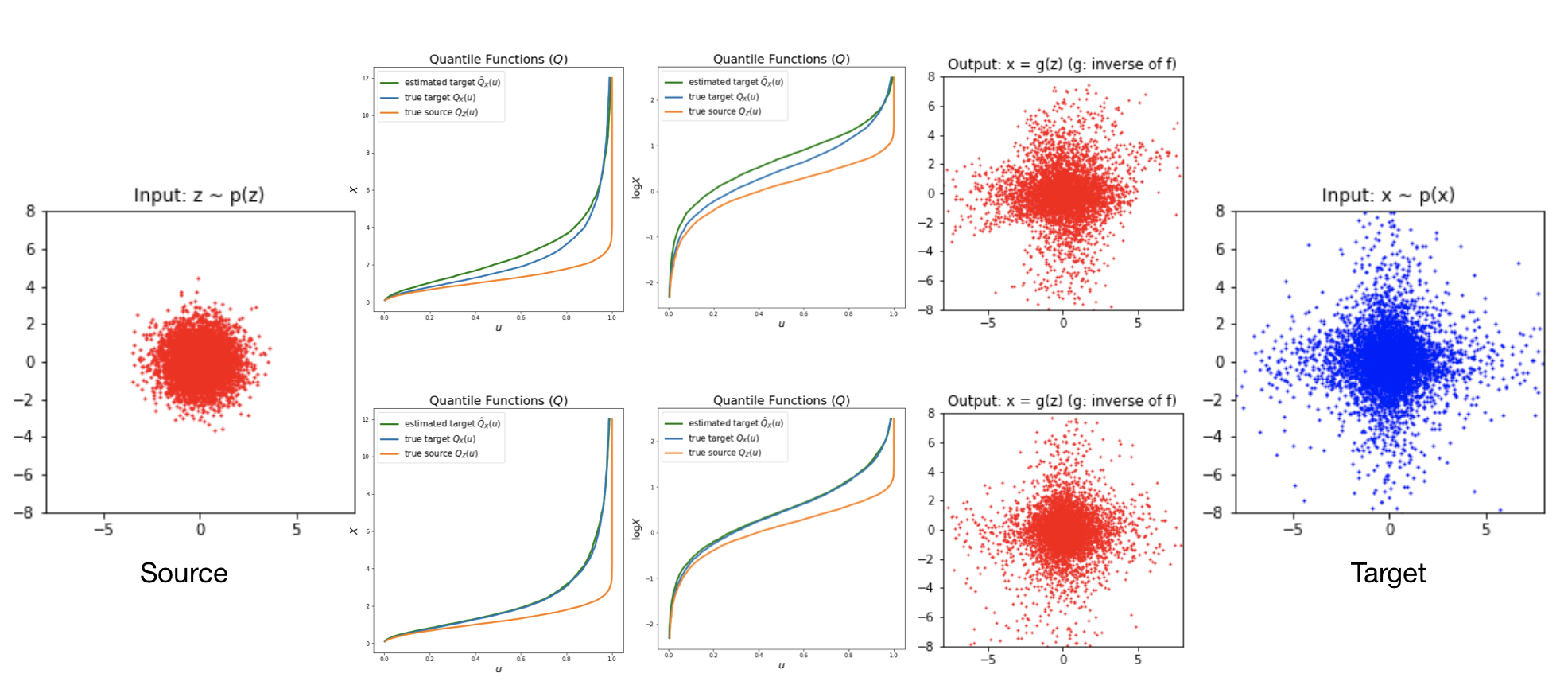}
    \caption{Results for SOS-Flows with degree of polynomial $r=2$ for two and three blocks. The first and last column plots the samples from the source (Gaussian) and target (student-t) distribution respectively. The two rows from top to bottom in second-fourth columns correspond to results from transformations learned using two, and three compositions (or blocks). The second and third column depict the quantile and log-quantile (for clearer illustration of differences) functions of the source (orange), target (blue), and estimated target (green) and the fourth column plots the samples drawn from the estimated target density.  The estimated target quantile function matches exactly with the quantile function of the target distribution illustrating that the higher-order polynomial flows like SOS flows can capture heavier tails of a target. This is further reinforced by their respective tail-coefficients which were estimated to be $\gamma_{\mathsf{source}} = 0.15$, $\gamma_{\mathsf{target}} = 0.81$, $\gamma_{\mathsf{estimated-target, 2}} = 0.76$, $\gamma_{\mathsf{estimated-target, 3}} = 0.81$. Best viewed in color.}
    \label{fig:sos}
\end{figure*}

\end{document}

%% file: section/intro.tex
\section{Introduction}
\label{sec:intro}

Increasing triangular maps are a recent construct in probability theory that can transform any source density to any target density \cite{BogachevKM05}. The Knothe-Rosenblatt transformation  \cite{rosenblatt1952remarks, knothe1957contributions} gives an explicit version of an increasing triangular map that does the transformation. These triangular maps provide a unified framework \cite{jaini2019sum} to study popular neural density estimation methods like normalizing flows \cite{TabakVE10, TabakTurner13, RezendeMohamed15} and autoregressive models \cite{PapamakariosPM17, HuangKLC18, KingmaSJCSW16, UriaCGML16, LarochelleMurray11} which are tractable methods for explicitly modelling densities for high-dimensional datasets. Indeed, these methods have been applied successfully in several domains including natural images, videos, speech and audio synthesis, novelty detection, and natural language.  

This work studies the tail properties of a target density by characterizing the properties of the corresponding increasing triangular map required to push a tractable source density with known tails to the desired target density. We begin in \S\ref{sec:uni} by showing that, in one dimension, the density quantile functions of the source and target density characterize the slope of a (unique) increasing transformation. Furthermore, the asymptotic properties of the density quantile function allow us to give a granular characterisation of the \emph{degree of heaviness} of a distribution. We show that the degree of heaviness parameter of the source and target densities characterize the properties of the corresponding triangular map completely. We then give a precise rate at which an increasing transformation must grow in order to capture the tail behaviour of the target density by drawing connections between the degree of heaviness parameter and the existence of higher-order moments of the densities.

We generalize these results for higher dimensions in \S\ref{sec:mul} by showing that a Lipschitz-continuous transport map will always result in a target density with the same tail properties as the source, highlighting the trade-off between choosing an \emph{appropriate} source density and \emph{sufficiently complex} transport map to capture tails in a target density. Additionally, when the source and target densities are from the elliptical family, we show that the increasing triangular map from a light-tailed distribution to a heavy-tailed distribution must have all diagonal entries of the Jacobian unbounded. 

In \S\ref{sec:flow}, we discuss the implications of these results for a class of flow based models that we call \emph{affine} triangular flows which include NICE \cite{DinhKB15}, Real-NVP \cite{DinhSDB17}, MAF \cite{PapamakariosPM17}, IAF \cite{KingmaSJCSW16}, and Glow \cite{KingmaDhariwal18}. We show both theoretically and empirically that these models as originally implemented lack the ability to push a fixed source density to a target density with heavier tails. To circumvent these draw-backs of affine flows, we subsequently propose \emph{tail-adaptive flows} in \S\ref{sec:taf}, where the source density, instead of being fixed, is endowed with a learnable parameter that controls its tail behaviour and allows affine flows to capture tail properties of the target density. We illustrate these properties of tail-adaptive flows empirically and demonstrate their performance on benchmark datasets. 

\textbf{Contributions.} We summarize our main contributions as follows:   
\begin{itemize}
    \vspace{-0.6em}
    \setlength\itemsep{-0.35em}
    \item We show that density quantiles precisely capture the properties of a push-forward transformation. We use these to provide asymptotic rates for the slope of maps required to capture heavy-tailed behaviour.
    \item We show that Lipschitz push-forward maps cannot change the tails of the source density qualitatively. We thus reveal a trade-off between choosing a ``complex'' source density and an ``expressive'' transformation for representing heavy-tailed target densities.
    \item As a consequence, we show that several popular flow models as originally implemented lack the ability to capture heavier tailed density than the fixed source.
    \item We propose tail-adaptive flows that can be deployed easily in any existing flow based and autoregressive model to better capture tail properties of a target density. We also demonstrate the importance of choosing an appropriate source density.
\end{itemize}
Due to space constraints, proofs are deferred to \Cref{sec:proofs}.

%% file: section/prelim.tex
\section{Preliminaries and Set-Up}
\label{sec:prelim}
In this section we set up our main problem, introduce key definitions and notations, and formulate the framework of characterizing tail properties of a target probability density through the unique triangular push-forward map.

We call a  mapping $\Tb: \RR^d \to \RR^d$ \emph{triangular} if its $j$-th component $T_j$ only depends on the first $j$ variables $z_1, \ldots, z_j$. The name ``triangular'' comes from the fact that the Jacobian $\grad \Tb$ is a triangular matrix function. Further, we call $\Tb$ increasing if for all $j \in [d]$, $T_j$ is an increasing function of $z_j$. Triangular transformations are appealing due to the following result by \citet{BogachevKM05}: 
\begin{theorem}[\citealt{BogachevKM05}]
	\label{thm:tri}
	For any two densities $p$ and $q$ over $\Zsf = \Xsf=\RR^d$, there exists a unique (up to null sets of $p$) \emph{increasing} triangular map $\Tb:\Zsf \to \Xsf$ so that if $\Zv \sim p$ then $\Tb(\Zv) \sim q$, i.e. $q$ is the push-forward of $p$, or in symbols $q = \Tbpush p$. 
\end{theorem}
Let us give an example to help understand \Cref{thm:tri}.

\begin{example}[Increasing Rearrangement]
\label{exp:ir}
Let $p$ and $q$ be univariate probability densities with distribution functions $F$ and $G$, respectively. One can define the increasing map $T = G^{-1} \circ F$ such that $q = \Tpush p$, where $G^{-1}: [0,1] \to \RR$ is the quantile function of $q$:
\begin{align}
G^{-1}(u) := \inf\{ t: G(t) \geq u \}.
\end{align}
Indeed, if $Z \sim p$, one has that  $F(Z)\sim \mathrm{uniform}$. Also, if $U\sim \mathrm{uniform}$, then   $G^{-1}(U) \sim q$. \Cref{thm:tri} is a rigorous iteration of this univariate argument by repeatedly conditioning (a construction popularly known as the Knothe-Rosenblatt transformation \cite{rosenblatt1952remarks, knothe1957contributions}). Specifically, the $j$-th component $T_j$ of $\Tb$ for the Knothe-Rosenblatt transformation is given by $x_j = T_j(z_1, \ldots, z_{j-1}, z_j) = F_{q, j|<j}^{-1} \circ F_{p, j|<j} (z_j)$ where $F_{q, j|<j}$ is the cdf of the conditional distribution of $\Xsf_j$ given $\Xsf_{<j} := (\Xsf_1, \ldots, \Xsf_{j-1})$, and similarly for $F_{p, j|<j}$.
\label{exm:univ}
\end{example}

\citet{jaini2019sum} showed that several popular normalizing flows and autoregressive models like NICE \cite{DinhKB15}, Real-NVP \cite{DinhSDB17}, IAF \cite{KingmaSJCSW16}, MAF \cite{PapamakariosPM17}, NAF \cite{HuangKLC18}, and SOS Flows \cite{jaini2019sum} employ increasing triangular transformations as fundamental modules to construct expressive push-forward transformations and are precisely special cases of learning increasing triangular maps.\footnote{We direct the reader to Section 3 and Table 1 in \citet{jaini2019sum} for a comprehensive overview of connecting triangular maps to several models in unsupervised learning.} 

In this work, we characterize the properties of increasing triangular maps required to capture the tail properties of the target density $q$ given a known source density $p$ and discuss the implications of these results for flow based models that use affine triangular transformations \eg Real-NVP \cite{DinhSDB17}, MAF \cite{PapamakariosPM17}, Glow \cite{KingmaDhariwal18}, etc. 

Formally, we characterize the tail properties of the target density $q$ by studying the properties of the induced increasing triangular map $\Tb$ acting on a known fixed source density $p$. This approach has been used earlier by \citet{SpantiniBM18} who studied the Markov properties of the target density and the existence of low-dimensional couplings by characterizing the properties of the induced triangular map and showing that such a map is both sparse and decomposable. Similar studies have been undertaken to characterize the tail properties of ``optimal'' transport maps by \citet{de2018tails} whose results only apply to a related limiting density but not the original ones, 
and for elliptical distributions by \citet{ghaffari2018multivariate}.
In contrast, we focus specifically on \emph{triangular maps} that are used extensively for tractable density estimation in normalizing flows and auto-regressive models \cite{jaini2019sum} and can be learned efficiently using deep neural networks. 

\Cref{exp:ir} shows that the increasing triangular map between two densities can be constructed iteratively by using the univariate increasing rearrangement repeatedly on the conditional distributions and the quantile functions. We employ the same strategy to characterize the properties of a triangular map $\Tb$ by characterizing the properties of the univariate maps $T_j$. Thus, in the next section, we first explore in detail the properties of univariate (increasing) maps.

%% file: section/uni.tex
\section{Properties of Univariate Transformations}
\label{sec:uni}

We define the class of heavy tailed distributions $\Hc$ as those that have no finite higher-order moments \cite{foss2011introduction}:
\begin{align*}
    \Hc := \Big\{ p: \forall~ \lambda >0, ~ m_p(\lambda) :=\underset{Z \sim p}{\EE} [e^{\lambda Z}] = \infty \Big\}.
\end{align*}
otherwise, it is light-tailed \ie  $p \in \Lc$ if all its higher-order moments are finite\footnote{We note that this definition is restricted to only right-tails. For the sake of simplicity we develop our results for right-tails, but they generalise  to left-tails naturally.}. We show that \emph{any} diffeomorphic transformation $T$ that pushes a source density $p \in \Lc$ to a target density $q \in \Hc$ cannot have a bounded slope globally. \begin{restatable}{theorem}{qslope}
\label{thm:qslope}
Let $p \in \Lc$ and $q \in \Hc$ such that $q = T_{\#}p$, where $T$ is a diffeomorphism. Then, 
for all $M > 0$ and all $z_0 > 0$ there is $z > z_0$, such that $T'(z) > M$. Conversely, if $T$ is a Lipschitz-continuous map \& $p \in \Lc$, then, $T_{\#}p \in \Lc$.
\end{restatable}

\begin{figure*}[!h]
    \centering
    \includegraphics[height=0.6\textwidth, scale=0.5]{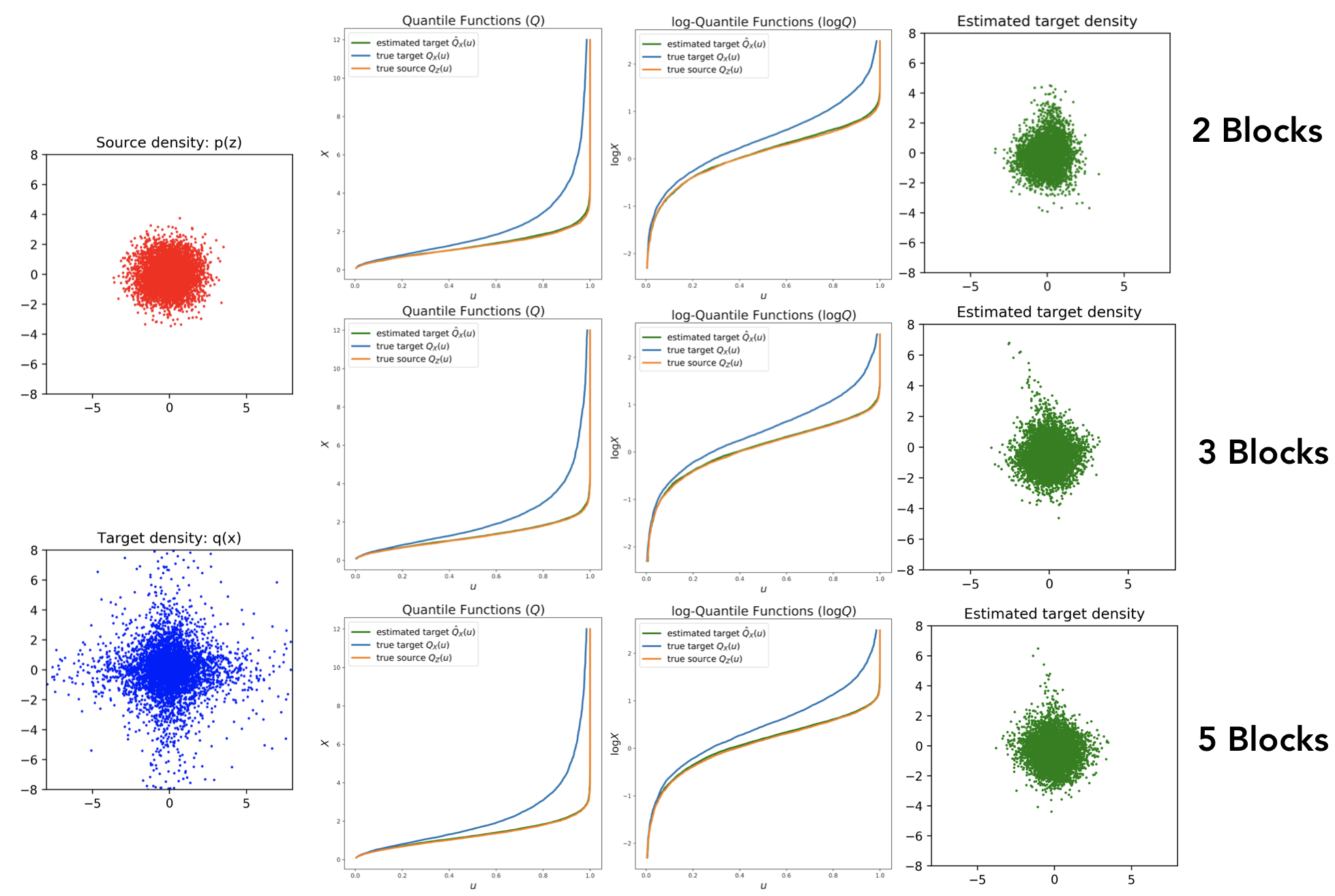}
    \vspace{-0.5em}
    \caption{Results for Real-NVP illustrating the inability to capture tails. The second and third column show the quantile and log-quantile plots for the source, target, and estimated target density. The quantile function of the source and the estimated target density are identical depictng the inability to capture heavier tails. This is further explained by the estimated tail-coefficients $\gamma_{\mathsf{source}} = 0.15$, $\gamma_{\mathsf{target}} = 0.81$, and $\gamma_{\mathsf{estimated-target}} = 0.15$. Best viewed in color. More details in \Cref{sec:flow}.}
    \label{fig:sig-nvp}
    \vspace{-1em}
\end{figure*}

\Cref{thm:qslope} is mostly a qualitative result, and it provides little knowledge about the map $T$ required to capture a heavy-tailed distribution $q$ given a source density $p$. Moreover, we would ideally like to characterize the properties of $T$ in terms of the ``\emph{degree of heaviness}'' of $p$ and $q$ respectively. We will address this problem by proposing a refined definition of tails of a density function in terms of the asymptotic behaviour of the density quantile function as formulated by \citet{parzen1979nonparametric} and \citet{andrews1973general}. 

For a probability density $p$ over a domain $\Zsf \subseteq \RR$, let $F_p : \Zsf \to [0,1]$ denote the cumulative distribution function of $p$, and $Q_p : [0,1] \to \Zsf$ be the quantile function given by $Q_p = F^{-1}_p$. Then, $fQ_p : [0,1] \to \RR_+ $ is called the density quantile function and is given by $fQ_p = \nicefrac{1}{Q^{'}_p}$. \citet{parzen1979nonparametric} proved that the limiting behaviour of any density quantile function as $u \to 1^-$ is given by:
\begin{align}
    \label{eq:dqf}
    fQ(u) \sim (1-u)^{\alpha}, \qquad \alpha >0
\end{align}
where $g(u) \sim h(u)$ implies that $ \lim_{u \to 1^{-}} \nicefrac{g(u)}{h(u)}$ is a finite constant. We can additionally define the limiting behaviour of the quantile function $Q(u)$ when $u \to 1^{-}$ as: 
\begin{align}
\label{eq:qf}
    Q(u) \sim (1-u)^{-\gamma}, \qquad \gamma = \alpha - 1.
\end{align}  
The parameter $\alpha$ is called the tail-exponent and defines the tail-area of a distribution and acts as a measure of ``degree of heaviness.'' Indeed, for two distributions with tail exponents $\alpha_1$ and $\alpha_2$, if $\alpha_1 > \alpha_2$, the former has heavier tails relative to the latter. Thus, the tail exponent $\alpha$ allows us to classify distributions based on their degree of heaviness.  
\begin{align*}
\label{eq: deght}
    \text{Define} \quad \Hc_{\alpha} := \Big\{ p ~:~ fQ_p \sim (1-u)^{\alpha}~ \text{as } u \to 1^{-}  \Big\}. 
\end{align*}
Following \citet{parzen1979nonparametric}, if $0 < \alpha < 1$ the distributions are light-tailed, \eg  the Uniform distribution. Here, we further show that a distribution has support bounded from above if and only if the right density quantile function has tail-exponent $0 < \alpha < 1$.
\begin{restatable}{proposition}{bsup}
\label{prop:comp_sup}
Let $p$ be a density with $fQ_p \sim (1-u)^{\alpha}$ as $u \to 1^{-}$. Then, $0 < \alpha < 1 ~ \textrm{iff} ~ \textrm{supp}(p) = [a, b]$ where $b < \infty$ \ie $p$ has a support bounded from above. 
\end{restatable}
$\Hc_{1}$ corresponds to a family of distributions for which all higher order moments exist. However, these distributions are relatively heavier tailed than short-tailed distributions and were termed as medium tailed distributions by \citet{parzen1979nonparametric}, \eg normal and exponential distribution. Additionally, for $\alpha=1$, a more refined description of the asymptotic behaviour of the quantile function can be given in terms of the shape parameter $\beta$:
\begin{align*}
    fQ(u) &\sim (1-u)\Big(\log\frac{1}{1-u}\Big)^{1-\beta}, \\ \text{and} \quad Q(u) &\sim \Big(\log\frac{1}{1-u}\Big)^{\beta}, \qquad 0 \le \beta \le 1.
\end{align*}
$\beta$ determines the degree of heaviness in medium tailed distributions; the smaller the value of $\beta$, the heavier the tails of the distribution \eg exponential distribution has $\beta = 1$, and normal distribution has $\beta = 0.5$. We thus define
\begin{align*}
    \Hc_{1, \beta} = \Big\{ p : fQ_p \sim  (1-u)\Big(\log\frac{1}{1-u}\Big)^{1-\beta}, ~0 \le \beta \le 1\Big \}
\end{align*}
and we have $\Hc_1 = \cup_{0 \le \beta \le 1} \Hc_{1, \beta}$. Further, the class of light tailed distributions defined in the beginning of the section is $\Lc = \cup_{0 < \alpha \le 1} \Hc_{\alpha}$. Finally, the class of heavy tailed distributions have $\alpha >1$ \ie $\Hc = \cup_{\alpha > 1} \Hc_{\alpha}$, \eg student-t distribution $t_{\nu}$ with $\nu$ degrees of freedom .

We are now in a position to characterize the map $T$ based on the degree of heaviness of the source and target densities. Following \Cref{exp:ir}, the slope of $T$ is given by the ratio of the density quantile function of the source and the target distribution respectively, \ie 
\begin{align*}
    T'(z) &= \frac{p(z)}{q\circ T(z)} = \frac{p\Big(F_p^{-1}\circ F_p(z)\Big)}{q\Big(F_q^{-1}\circ F_p(z)\Big)} \\
   \ie  \quad T'(z) &= \frac{fQ_p(u)}{fQ_q(u)}, \quad \text{where} \quad u = F_p(z).
\end{align*}
Clearly, the density quantile functions precisely characterizes the slope of an increasing map needed to push a source density $p$ to a target density $q$. 
\begin{proposition}
\label{thm:slp-q}
Let $p$ and $q$ be two square integrable univariate densities such that $q := T_{\#}p$. If the density quantile $fQ_p$ of $p$ shrinks to 0 at a rate slower than the density quantile $fQ_q$ of $q$, then $T'(z)$ is asymptotically unbounded.
\end{proposition}
\Cref{exm:norm2t1} in \Cref{app:ell} helps to illustrate \Cref{thm:slp-q} and the next corollary provides a precise characterization of asymptotic properties of a diffeomorphic transformation between densities with varying tail behaviour.

\begin{corollary}
\label{thm:uni}
Let $p \in \Hc_{\alpha_p}$ be a source density, $q \in \Hc_{\alpha_q}$ be a target density and $T$ be an increasing transformation such that $q= T_{\#}p$. Then, $\lim_{z \to \infty} T'(z) = \lim_{u \to 1^{-}} (1-u)^{\alpha_q - \alpha_p}$. Further, if $\alpha_p = \alpha_q = 1$, then $\lim_{z \to \infty} T'(z) = \lim_{u \to 1^{-}} \big(\log \nicefrac{1}{(1-u)}\big)^{\beta_p - \beta_q}$ where $u = F_p(z)$.
\end{corollary}

\Cref{exm:unif2norm} in \Cref{app:ell} further underlines the importance of density quantile functions to study tails of increasing transformations. We now connect the tail-exponent parameter $\alpha(\cdot)$ to the existence of higher-order moments of a random variable. Given a random variable $X \sim p$, the expected value of a function $g(x)$ can be written in terms of the quantile function as: $\mathbb{E}_{p}[g(x)] = \int_{0}^1 gQ_p(u) \rmd u$. This allows us to draw a precise connection between the degree of heaviness of a distribution as given by the density quantile functions (and tail exponent $\alpha$) and the the existence of the number of its higher-order moments ($\omega$).
\begin{restatable}{proposition}{exp}
\label{thm:qmom}
Let $p$ be a distribution with $Q_p(u) \sim (1-u)^{-\gamma}$ as $u \to 1^{-}$. Then,
$\int_{z_0}^\infty z^\omega p(z) dz  $ exists and is finite for some $z_0$ iff $\omega < \frac{1}{\gamma}$.
\end{restatable}
\begin{corollary}
If $p$ is a distribution with $Q_p(u) \sim (1-u)^{-\gamma}$ as $u \to 1^{-}$ and $Q_p(u) \sim u^{-\gamma}$ as $u \to 0^{+}$.\footnote{This condition takes the left-tail into account as well. Note that it is not necessary for both tails to have the same behaviour and our analysis  extends to such cases.} Then, $\mathbb{E}_p[|z|^{\omega}]$ 
exists and is finite iff $\omega < \frac{1}{\gamma}$.
\end{corollary}
Based on these observations, we can equivalently define heavy-tailed distributions as follows:
\begin{definition}
A distribution $p(z)$ with compact support \ie $\textrm{supp}(p) = [a,b]$ where $|a| < \infty$ and $|b| < \infty$ is said to be $\omega-$heavy tailed if for all $0 < \mu < \omega$, $\mathbb{E}_{p}[|z-b|^{\nicefrac{1}{\mu}}]$ exists and is finite, but for $\mu \ge \omega$,  $\mathbb{E}_{p}[|z-b|^{\nicefrac{1}{\mu}}]$ is infinite or does not exist. 
\end{definition}

\begin{definition}
\label{def:alpha1}
A distribution $p(z)$ with tail exponent $\alpha = 1$ is said to be $\omega-$heavy tailed if for all $0 < \mu < \omega$, $\mathbb{E}_{p}[e^{|z|^{\mu}}]$ exists and is finite, but for $\mu \ge \omega$,  $\mathbb{E}_{p}[e^{|z|^{\mu}}]$ is infinite or does not exist.
\end{definition}

\begin{definition}[$\omega^{-1}-$heavy tailed distributions]
\label{def:ohv}
A distribution $p(z)$ with tail-exponent $\alpha >1$ is heavy tailed with degree $\omega^{-1}$ with $\omega \in \RR_+$ if for all $0<\mu < \omega$, $\mathbb{E}_p[|z|^{\mu}]$ exists and is finite, but for all $\mu \geq \omega$, $\mathbb{E}_p[|z|^{\mu}]$ is infinite or does not exist.
\end{definition}
These definitions allow us to finally give the rate an increasing transformation must emulate to exactly represent tail-properties of a target density given some source density. 
\vspace{-0.5em}
\begin{restatable}{proposition}{rate}
\label{thm:rate}
Let $p$ be a  $\omega_p^{-1}-$heavy distribution, $q$ be a $\omega_q^{-1}-$heavy distribution  and $T$ be a diffeomorphism such that $q := T_{\#}p$. 
Then for small $\epsilon > 0$, $T(z) = o(|z|^{\nicefrac{\omega_p}{\omega_q - \epsilon} })$.
\end{restatable}

%% file: section/mul.tex
\section{Properties of Multivariate Transformations}
\label{sec:mul}

We now generalize our results to higher dimensions by first fixing the definition of a heavy-tailed distribution in higher dimensions \footnote{Note that due to the lack of total ordering there is no standard definition of multivariate heavy-tailed distributions.}. We say that a random variable $\Xsf \subseteq \RR^d$ admits a heavy-tailed density function if the univariate random variable $\|\Xsf\|$ has a heavy tailed density where $\|\cdot\|$ is some norm function. The granular definitions from \Cref{sec:uni} can be extended to the multivariate case through the density function of $\|\Xsf\|$.

\begin{restatable}{theorem}{mul}
\label{thm:mul-ubdd}
Let $\Zsf \subseteq \RR^d$ be a random variable with density function $p$ that is light-tailed and $\Xsf \subseteq \RR^d$ be a target random variable with density function $q$ that is heavy-tailed. Let $\Tb : \Zsf \to \Xsf$ be such that $q = \Tb_{\#}p$, then $\Tb$ cannot be a Lipschitz function. 
\end{restatable}

\begin{restatable}{corollary}{cormul}
Under the same set-up as in \Cref{thm:mul-ubdd}, there exists an index $i \in [d]$ such that $\|\nabla_{\zv} T_i\|$ is unbounded. 
\end{restatable}

\Cref{thm:mul-ubdd} is a general result for any diffeomorphic transformation between two densities and we discuss the implication of this result for flow based models in \S\ref{sec:flow}. However, before proceeding further, we also characterize the properties of the triangular map $\Tb$ such that $q = \Tbpush p$ by studying the properties of the univariate maps $T_j, ~j\in [d]$ obtained by repeated conditioning when the source and target densities are from the class of elliptical distributions.
\begin{definition}[Elliptical distribution, \cite{cambanis1981theory}]
\label{def:elliptical}
A random vector $\Xsf \subseteq \RR^d$ is said to be elliptically distributed denoted by $\Xsf \sim \veps_d(\muv, \Sigma, F_{R})$ with $\rank(\Sigma) = r$ if and only if  there exists a $\muv \in \RR^d$, a matrix $\Av \in \RR^{d \times r}$ with maximal rank $r$, and a non-negative random variable $R$, such that $\Xsf \stackrel{d}{=} \muv + R \Av \Uv^{(d)}$, where the random $r$-vector $\Uv$ is independent of $R$ and is uniformly distributed over the unit sphere $\Bc_{d-1}$, $\Sigma = \Av^{T} \Av$ and $F_R$ is the cumulative distribution function of the variate $R$.
\end{definition}

For ease in developing our results, we consider only full rank elliptical distributions \ie $\rank(\Sigma) = d$ but the results can be easily extended to the general case. The spherical random vector $U^{(d)}$ produces elliptically contoured density surfaces due to the transformation $\Av$. The density function of an elliptical distribution as defined above is given by:
 $f(x) =    |\det \Sigma |^{-\frac{1}{2}} g_R \big((x-\mu)^T\Sigma^{-1} (x-\mu)\big)$, 
where the function $ g_R(t): [0, \infty) \to [0, \infty) $  is related to $f_R$, the density function of $R$,  by the equation: $f_R(r) = s_d  r^{d-1} g_R(d^2), \ \forall d \ge 0 $, here $s_d = \frac{2\pi^{d/2} }{\Gamma(d/2)}$ is the area of a unit sphere. Thus, the tail properties of a random variable $\Xsf$  with an elliptical distribution $\veps_d(\muv, \Sigma, F_R)$ is determined by the generating random variable $R$. Indeed, $\Xsf$ is heavy-tailed in all directions if the univariate generating random variable $R$ is heavy-tailed. 

Define 
$ m_{f_{R}}(k) =  \frac{1}{s_d}  \int_0^\infty r^k f_R(r)\rmd r, \quad \forall ~ k \in \RR_+$. 
Intuitively, $m_{f_{R}}(k)$ is the $k$-th order moment of $f_R$ when $k$ is integer-valued. This allows us to generalize the granular definition of heavy-tailed distributions (\S\ref{sec:uni}, \Cref{def:ohv}) to the multivariate elliptical case: the distribution
$\veps_d(\muv, \Sigma, F_{R})$ is $\omega^{-1}$-heavy iff  $\mu_k$ is finite for all $k < \omega$ \ie iff $F_R$ is $\omega^{-1}$-heavy. Similarly, from Definition \ref{def:alpha1} one has that $\veps_d(\muv, \Sigma, F_{R})$ is $\omega$-heavy iff  $F_R$ is $\omega$-heavy. Elliptical distributions have certain convenient properties: marginal, conditional and linear transformation of an elliptical distribution are also elliptical (see \Cref{app:ell}). 
Furthermore, we derive the degree of heaviness parameter of the  conditional distributions of an elliptical distribution.
\begin{restatable}{proposition}{elc}
\label{prop:elc}
Under the same assumptions as in \Cref{lemma:ell_con} (App.\ref{app:ell}), if $\Xsf \sim \veps_d(0, \Iv, F_R)$ is $\omega^{-1}$-heavy, then the conditional distribution of $\Xsf_2|(\Xsf_1 = \xv_1)$ is $(\omega + d_1)^{-1}$-heavy where $\Xsf_1 \subseteq \RR^{d_1}$.
\end{restatable}
Equipped with all the necessary results, we now show that an increasing triangular map $\Tb$ between a light-tailed and a heavy-tailed elliptical distribution has all diagonal entries of $\grad \Tb$ unbounded. 
\begin{restatable}{proposition}{mult}
Let $\Zsf \sim \veps_d(0, \Iv, F_{S})$ and $\Xsf \sim \veps_d(0, \Iv, F_{R})$ have densities $p$ and $q$ respectively where $F_R$ is heavier tailed than $F_S$. If $\Tb : \Zsf \to \Xsf$ is an increasing triangular map such that $q := \Tbpush p$, then all diagonal entries of $~\grad \Tb$ and $\mathsf{det}|\grad \Tb|$ are unbounded. 
\end{restatable}

\begin{remark}
Our analysis naturally extends to the case when the target density is lighter tailed by studying the corresponding inverse transformation $\Tb^{-1}$. Particularly, such a transformation should have a vanishing asymptotic slope to capture lighter-tailed distributions. 
\end{remark}

%% file: section/flow.tex
\section{(Lack of) Tails in Affine Flows}
\label{sec:flow}
We call a triangular map $\Tb$ affine if $T_j(z_j ; z_1, \cdots, z_{j-1})$ is an affine function of $z_j$.
Several autoregressive and flow models like NICE \cite{DinhKB15}, Real-NVP \cite{DinhSDB17}, IAF \cite{KingmaSJCSW16}, MAF \cite{PapamakariosPM17}, and Glow \cite{KingmaDhariwal18} use affine triangular maps as fundamental building blocks to construct expressive transport maps through composition (see~\Cref{tab: flows}).
\begin{table}[t]
\caption{Affine triangular flows}
\vspace{1em}
\centering
\scalebox{0.87}{
\begin{tabular}{c|c|c}
\toprule
 Model & coefficients & $T_j\big(z_j~; z_1, \ldots, z_{j-1}\big)$  \\[3pt]
 \midrule
 NICE  &  $\mu_j(z_{<l})$& $z_j + \mu_j\cdot \mathbf{1}_{j \not\in [l]}$ \\[3pt]
IAF  & $\sigma_j(z_{<j}), ~\mu_j(z_{<j})$  & $\sigma_j z_j + (1-\sigma_j)\mu_j$ \\[3pt]
 MAF & $\lambda_j(z_{<j}), ~\mu_j(z_{<j})$  & $z_j \cdot\mathsf{exp}(\lambda_j) + \mu_j$ \\[3pt]
 Real-NVP & $\lambda_j(z_{<l})$, $\mu_j(z_{<l})$ & $\mathsf{exp}(\lambda_j \cdot  \mathbf{1}_{j \not\in [l]}) \cdot  z_j  + \mu_j \cdot  \mathbf{1}_{j \not\in [l]}$ \\[3pt]
  Glow & $\sigma_j(z_{<l})$, $\mu_j(z_{<l})$ & $\sigma_j \cdot  z_j  + \mu_j \cdot  \mathbf{1}_{j \not\in [l]}$ \\[3pt]
 \bottomrule
\end{tabular}}
\label{tab: flows}
\end{table}
In the aforementioned models, the coefficients in \Cref{tab: flows} are the output of another network such that $\lambda_j = \mathsf{sigmoid}\big(f(z_1, \cdots, z_{j-1})\big)$ or $\lambda_j = \mathsf{tanh}\big(f(z_1, \cdots, z_{j-1})\big)$\footnote{IAF and Glow use $\sigma_j = \mathsf{sigmoid(\cdot)}$} and $\mu_{j} = \mathsf{relu}\big(g(z_1, \cdots, z_{j-1})\big)$, resulting in transformations that lack the ability to learn target densities with heavier tails than the source density.
We formalize this result below.
\begin{restatable}{theorem}{Flow}
\label{thm:Flow}
Let $p$ be a light-tailed density and $\Tb$ be a triangular transformation such that $T_j(z_j; ~z_{<j}) = \sigma_{j}\cdot z_j + \mu_j$. If, $\sigma_j(z_{<j})$ is bounded above and $\mu_j(z_{<j})$ is Lipschitz then the target density $q := \Tbpush p$ is light-tailed. 
\vspace{-0.5em}
\end{restatable}
Conversely, if $\sigma_j(z_{<j})$ is bounded and $\mu_j(z_{<j})$ is Lipschitz then the tails of $q$ can not be heavier than the source density $p$.
Moreover, since Lipschitz-continuous affine transformations cannot change tail behaviour, linear maps like permutations and 1$\times$1 convolutions will also lack the ability to capture heavier tails.
Thus, through an iterative argument, it is seen easily that composition of several such affine triangular maps (combined with permutations and 1$\times$1 convolutions) will still be unable to push a source density to a target density with heavier tails.
We note that most models in \Cref{tab: flows} in their proposed functional form can capture heavier tails due to the exponential term in the coefficient.
However, we found that in practice this leads to instability during training while capturing heavy-tailed distributions.
Instead, $\mathsf{sigmoid}(\cdot), \mathsf{tanh}(\cdot), \text{ and }  \mathsf{relu}(\cdot)$ are used for modeling these affine flows resulting in models that are unable to capture heavier tails.

We next use synthetic experiments to supplement our findings above and illustrate this inability of certain affine flows to capture tails empirically.
In our setup, we choose a bi-variate Gaussian distributed random variable \ie $\Zsf \sim \Nc(0, \mathbf{I})$ as the source and a bi-variate student-t distributed target random variable with two degrees of freedom \ie $\Xsf \sim t_{2}(0, \mathbf{I})$. We measure the tail behaviour of a multivariate random variable by measuring the tail-coefficient $\gamma$ (c.f. Eq.\eqref{eq:qf}) of the quantile function of the $\ell_2$ norm of the random variable.  We recall that if the tail-coefficient of a density $q$ is larger than another density $p$, then $q$ is heavier tailed than $p$ (see \S\ref{sec:uni}). We learn an affine triangular flow $\Tb : \Zsf \to \Xsf$ where we experiment with architectures  the same as Real-NVP, MAF, and Glow. We generated 10,000 samples from the target density and used negative 
log-likelihood as the training objective. We divided the dataset into training-validation-testing in the ratio 2:1:1. We trained the model using Adam \cite{kingma2014adam} for 40 epochs with a batch size of 128 and learning rate of $10^{-3}$.  

\Cref{fig:sig-nvp} shows the results in detail for Real-NVP with $\lambda(\cdot) = \mathsf{tanh}(\cdot)$. The first column plots the samples from the source (Gaussian, red) and target (student-t, blue) distribution, respectively. The three rows from top to bottom in second to fourth columns correspond to results from transformations learned using two, three, and five compositions (or blocks). The second and third column depict the quantile and log-quantile (for clearer illustration of differences) functions of the source (orange), target (blue), and estimated target (green) and the fourth column plots the samples drawn from the estimated target density.  The estimated target quantile function matches exactly with the quantile function of the source distribution illustrating the inability of Real-NVP to capture tails. This is further reinforced by the tail-coefficients $\gamma_{\mathsf{source}} = 0.15$, $\gamma_{\mathsf{target}} = 0.81$, and $\gamma_{\mathsf{estimated-target}} = 0.15$. The negative log-likelihoods for the target, and the estimated target on test data were $-3.95$ and $-3.82$ respectively. We also observe that the samples generated from the estimated target density capture only the high density regions of the target but fail to spread to the tail regions of the target density. We show similar results for the quantile and log-quantile plots in \Cref{fig:tanh-nvp} when $\lambda(\cdot) = \mathsf{sigmoid}(\cdot)$. In \Cref{fig:rmg} we show the results for architectures using MAF, Glow, and Real-NVP with composition of 5 blocks. 

\begin{figure}[t]
    \centering
    \includegraphics[width=0.42\textwidth]{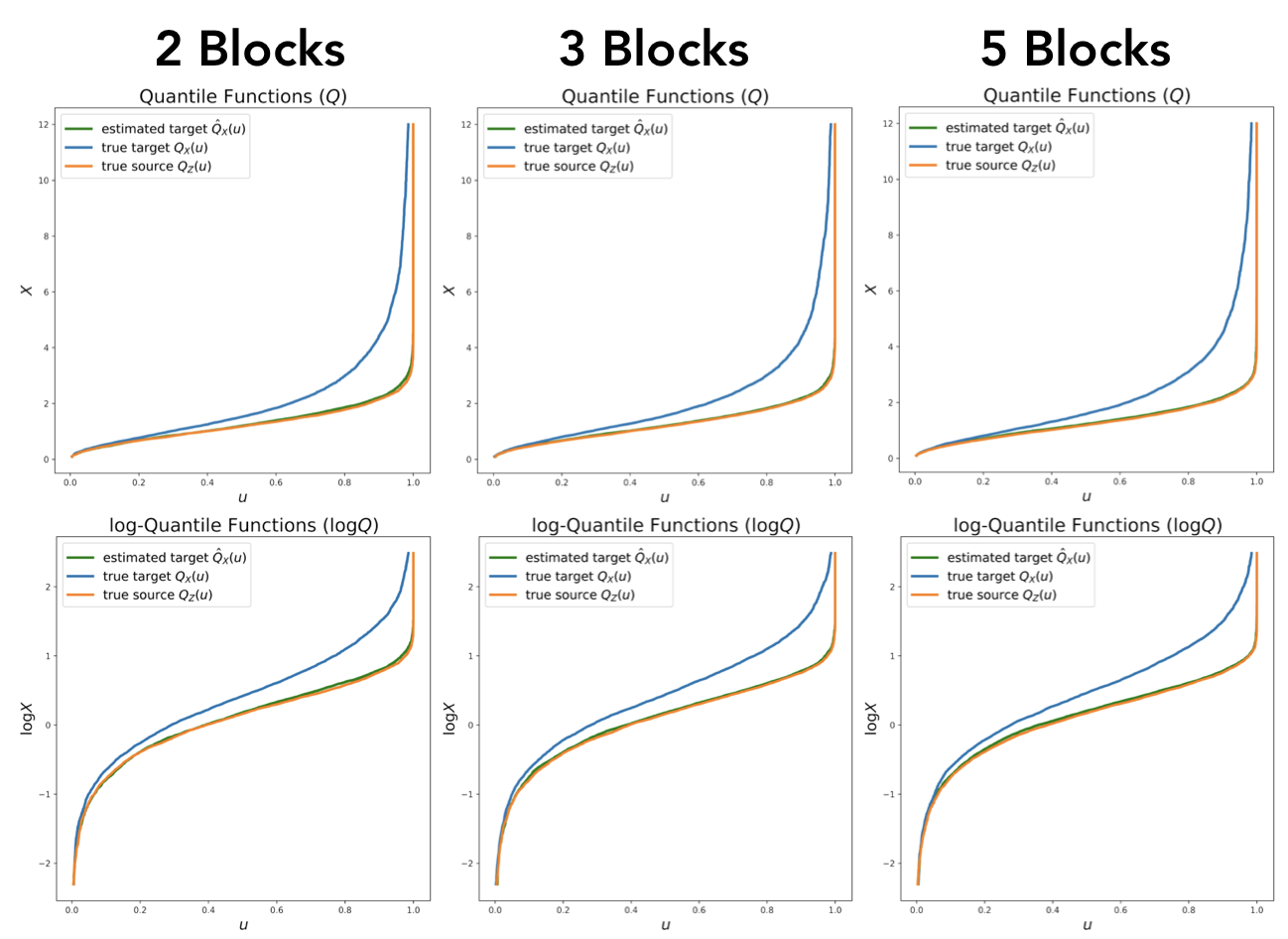}
    \caption{Same as \Cref{fig:sig-nvp} with $\lambda(\cdot) = \mathsf{sigmoid}(\cdot)$. The three columns correspond to two, three, and five compositions (blocks). $\gamma_{\mathsf{source}} = 0.15$, $\gamma_{\mathsf{target}} = 0.81$, and $\gamma_{\mathsf{estimated-target}} = 0.15$.}
    \label{fig:tanh-nvp}
\end{figure}

\begin{figure}[!t]
    \begin{subfigure}
        \centering
    \includegraphics[height= 0.13\textwidth, keepaspectratio]{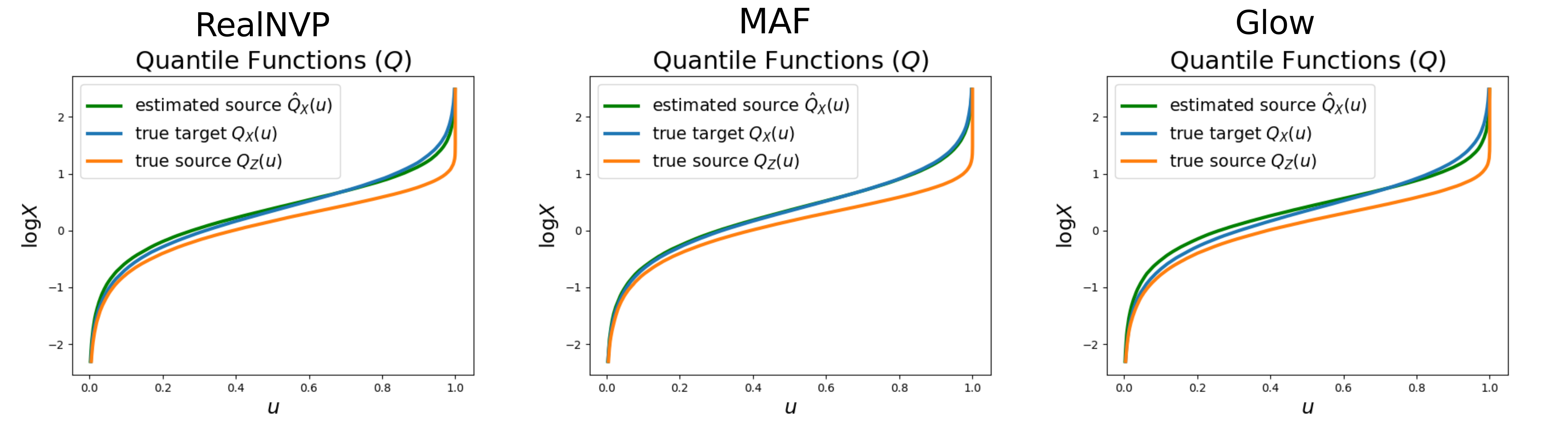}
    \end{subfigure}
    \begin{subfigure}
    \centering
    \includegraphics[height= 0.13\textwidth, keepaspectratio]{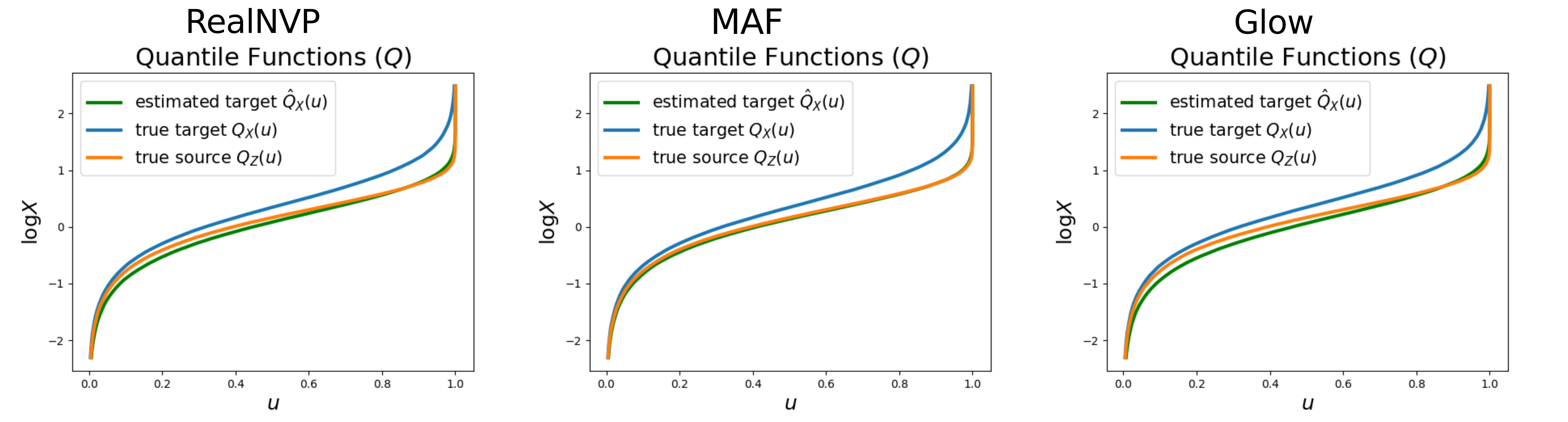}
    \end{subfigure}
    \caption{Quantile functions of distributions. Source is i.i.d student-t with 3 degree of freedom, true target is Gaussian, estimated source is a distribution modelled by an affine flow (in generative direction). Left is  RealNVP, Middle is MAF, Right is Glow. Bottom row corresponds to the setting of $\Tb : \Zsf \to \Xsf$ and top row corresponds when $\Tb^{-1}:\Xsf \to \Zsf$}
    \label{fig:rmg}
    \vspace{-1em}
\end{figure}

\section{Tail-Adaptive Flows}
\label{sec:taf}

\begin{table*}[!h]
\begin{center}
\caption{Average test log-likelihoods and standard deviation for Tail-adaptive flows (TAF) over 5 trials (higher is better). The numbers in the parenthesis indicate the number of compositions used. Results for other models are from \cite{HuangKLC18, jaini2019sum}.}
\label{tab:real}
\vspace{1ex}
\begin{tabular}{c|c|c|c|c|c} 
\toprule
Method & Power & Gas & Hepmass& MiniBoone & BSDS300\\
\midrule
 MADE  &0.40 $\pm$ 0.01    &8.47 $\pm$ 0.02   &-15.15 $\pm$ 0.02 & -12.24 $\pm$ 0.47 & 153.71 $\pm$ 0.28 \\
 MAF affine (5) & 0.14 $\pm$ 0.01  & 9.07 $\pm$ 0.02   &-17.70 $\pm$ 0.02 & -11.75 $\pm$ 0.44&155.69 $\pm$ 0.28\\
 MAF affine (10) &0.24 $\pm$ 0.01 & 10.08 $\pm$ 0.02&  -17.73 $\pm$ 0.02 & -12.24 $\pm$ 0.45 & 154.93 $\pm$ 0.28\\
 MAF MoG (5)  &0.30 $\pm$ 0.01   & 9.59 $\pm$ 0.02 & -17.39 $\pm$ 0.02 & -11.68 $\pm$ 0.44 & 156.36 $\pm$ 0.28\\
 TAN & 0.60 $\pm$ 0.01  & 12.06 $\pm$ 0.02 &-13.78 $\pm$ 0.02 & -11.01 $\pm$ 0.48 & 159.80 $\pm$ 0.07\\
 NAF DDSF (5)  &0.62 $\pm$ 0.01 & 11.91 $\pm$ 0.13 & -15.09 $\pm$ 0.40 & -8.86 $\pm$ 0.15 & 157.73 $\pm$ 0.04\\
 NAF DDSF (10)&0.60 $\pm$ 0.02  & 11.96 $\pm$ 0.33 & -15.32 $\pm$ 0.23 &  -9.01 $\pm$ 0.01 & 157.43 $\pm$ 0.30\\
 SOS  (7) & 0.60 $\pm$ 0.01 & 11.99 $\pm$ 0.41 & -15.15 $\pm$ 0.10 & -8.90 $\pm$ 0.11& 157.48 $\pm$ 0.41\\
 \midrule
 TAF affine (5) & 0.28 $\pm$ 0.01 & 9.87 $\pm$ 0.23 & -17.41 $\pm$ 0.20 & -11.71 $\pm$ 0.09& 156.53 $\pm$ 0.52\\
TAF SOS (7) & 0.59 $\pm$ 0.01 & 11.99 $\pm$ 0.34 & -15.11 $\pm$ 0.18 & -8.94 $\pm$ 0.23& 157.52 $\pm$ 0.22\\
  \bottomrule
\end{tabular}
\end{center}
\end{table*}

We saw in Sections~\ref{sec:uni} and \ref{sec:mul} that a Lipschitz-continuous map cannot push-forward a light-tailed source density to heavier tailed target density. Subsequently, we illustrated in \Cref{sec:flow} that several flow models that incorporate compositions of triangular affine maps as the function class for the transport map are unable to capture densities that are heavier tailed than the chosen source density. In \Cref{fig:sos} in \Cref{app:ell} we also show that for the same experiment set-up where affine triangular flows were unable to capture heavier tails of a density,  SOS flows \cite{jaini2019sum} that use higher-order polynomial maps were able to learn the heavy-tail properties. These findings demonstrate a trade-off between choosing a \emph{complex source density} vs.\ \emph{expressive transformations}. Intuitively, following \Cref{thm:uni} it is clear that a Lipschitz map is appropriate to learn tails of a target density if both the source distribution and the target distribution belong to a family of densities that have equally heavy tails. However, if the two densities are from families with differing degree of heaviness then the transformation needs to be more expressive than a Lipschitz-continuous function.
This choice of either using source densities with the same heaviness as the target, or deploying more expressive transformations than Lipschitz functions is what we refer to as the trade-off  between choosing a \emph{complex source density} vs. \emph{expressive transformations}. 

In practice, however, we do not know a priori the degree of heaviness of a target distribution to guide the choice of the source density accordingly. We circumvent this problem by proposing tail-adaptive flows (TAFs) wherein the tail property of the source density can be adapted during training such that simpler transformations like Lipschitz maps are able to capture heavy-tailed target distributions. In our approach, we propose to fix the source density as a standard student-t distribution with its degrees of freedom being a learnable parameter \ie $\Zsf \sim t_{\nu}(0, \mathbf{I})$ where $\nu \in (1, \infty)$ is a learnable parameter. The source density becomes lighter tailed as $\nu$ increases and approaches a Gaussian distribution as $\nu \to \infty$. The source density is still tractable and hence we can learn the transport map $\Tb$ and degrees of freedom $\nu$ by maximizing the likelihood of the target density. 

We thus formulate the density estimation paradigm for tail-adaptive flows as follows: Suppose we have access to an \iid sample $\lbag \xv_1, \ldots, \xv_n \rbag\sim q$ and our interest lies in estimating $q$ and capturing its tail behaviour. Let $\mathcal{F}$ be a class of mappings and $p_{\nu}$ be the source density which is a standard student-t distribution with $\nu$ degrees of freedom \ie $p_{\nu}:= t_{\nu}(0, \mathbf{I})$. The log-likelihood objective is:
\begin{multline*}
\max_{\substack{\Tb \in \mathcal{F}\\ \nu \in (1, \infty)}}~~ \frac{1}{n}\sum_{i=1}^n \Big[-\log |\Tb'(\Tb^{-1}\xv_i)| ~+~  \log p_{\nu}(\Tb^{-1}\xv_i)  \Big],
\end{multline*}
where $\log p_{\nu}(\Tb^{-1}\xv_i) = d\log \Gamma(\frac{\nu+1}{2}) - d\log \Gamma(\frac{\nu}{2}) - \frac{d}{2}\log \nu - \sum_{j=1}^{d}\frac{\nu+1}{2}\cdot\log \Big(1 + \frac{z_{ij}^2}{\nu}\Big)$, $z_{ij} = T_j^{-1}(x_{ij})$ and $\Gamma(\cdot)$ is the gamma function. Tail-adaptive flows are easy to implement as they can be easily optimized using automatic differentiation. Further, they can be plugged-in any existing flow based learning framework to substitute the Gaussian density since the transformation $\Tb$ in the objective above can be from any family of functions. In our experiments we used tail-adaptive flows with Real-NVP, MAF, and SOS flows to illustrate its performance on inference tasks on real datasets and ability to capture tails with affine flows on synthetic datasets.

We first show that affine tail-adaptive flows can capture heavier tailed distributions. In \Cref{fig:adap_t}, we give the results for tail-adaptive flows using Real-NVP on the synthetic experiment we used in \Cref{sec:flow}. We kept the set-up of the experiment exactly the same as before with the source distribution to be $t_{\nu}(0,\mathbf{I})$ and initialised $\nu=30$. It is evident from the figure that tail-adaptive flows are able to capture the heavy-tails since the density quantiles of the target and estimated target overlap with $\gamma_{\mathsf{source}} = 0.15$, $\gamma_{\mathsf{target}} = 0.81$, and $\gamma_{\mathsf{estimated-target}} = 0.80$. 

Next, we considered another setting to test the performance of tail-adaptive flows where we fixed the target density to be a bi-variate Neal's funnel distribution given by $\xv_i = (x_{1,i}, x_{2,i})$ where $x_{1,i} \sim \mathcal{N}(0, 1)$ and $x_{2,i} \sim \mathcal{N}(0, \mathsf{exp}\big(0.5x_{1,i})\big)$ and generated 10,000 samples from this distribution. We fixed the flow architecture to follow Real-NVP with $\lambda(\cdot) = \mathsf{tanh}(\cdot)$ and trained the model using Adam for 40 epochs with a batch size of 128 and learning rate of $10^{-3}$. We learned tail-adaptive flows with two, three, and five blocks respectively and the results are given in \Cref{fig:neal}. Here we noticed that as the number of blocks increased, the estimated target density approximated the true target density more faithfully. Furthermore, we also noticed that the tails became heavier as the number of stacked blocks increased with $\gamma_{\mathsf{source}} = 0.15$, $\gamma_{\mathsf{target}} = 0.63$, and $\gamma_{\mathsf{estimated-target, 2}} = 0.36$, $\gamma_{\mathsf{estimated-target, 3}} = 0.56$, and $\gamma_{\mathsf{estimated-target, 5}} = 0.61$.

Lastly, we replicate density estimation experiments on benchmark datasets popularly used to measure performance of flows and autoregressive models. Here we illustrate that tail-adaptive flows can be incorporated easily in existing architectures and achieve comparable performance on inference tasks. In \Cref{tab:real}, we report the performance of tail-adaptive flows using MAF \cite{PapamakariosPM17} and SOS \cite{jaini2019sum} keeping the architecture fixed as reported in the original papers but changing the source to tail adaptive ones. We compare the results to original implementations using Gaussian source density and other models like NAF \cite{HuangKLC18}, TAN \cite{OlivaDZPSXS18}, and MADE \cite{GermainGML15}. 

\begin{figure}[t]
\begin{subfigure}
    \centering
    \includegraphics[width=0.22\textwidth]{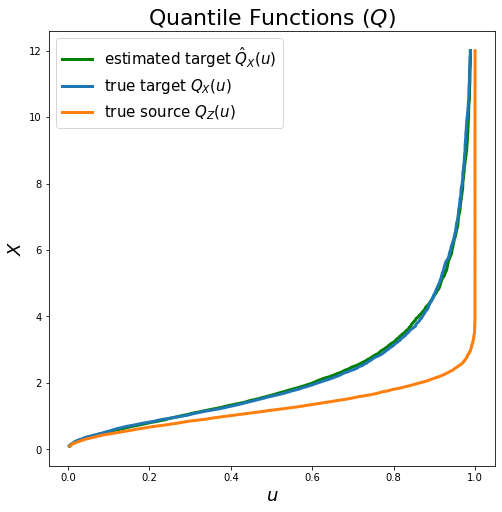}
\end{subfigure}
\begin{subfigure}
    \centering
    \includegraphics[width=0.22\textwidth]{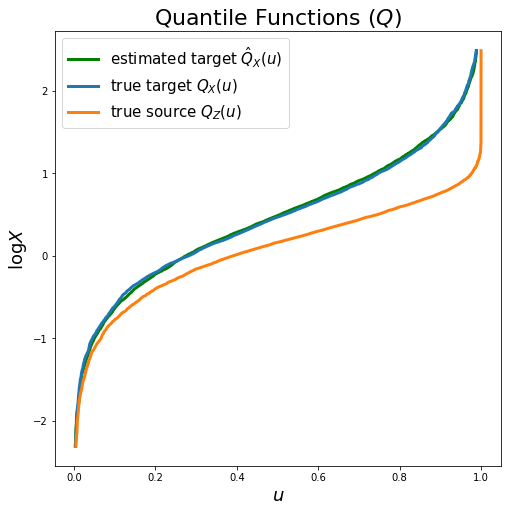}
\end{subfigure}
    \caption{Quantile and log-Quantile plots for 5 block Real-NVP with tail adaptive flows on the same setup as in \Cref{fig:sig-nvp}. Best viewed in color.}
    \label{fig:adap_t}
    \vspace{-1em}
\end{figure}

\begin{figure}[t]
    \centering
    \hspace*{-0.5cm}  
    \includegraphics[height=0.33\textwidth]{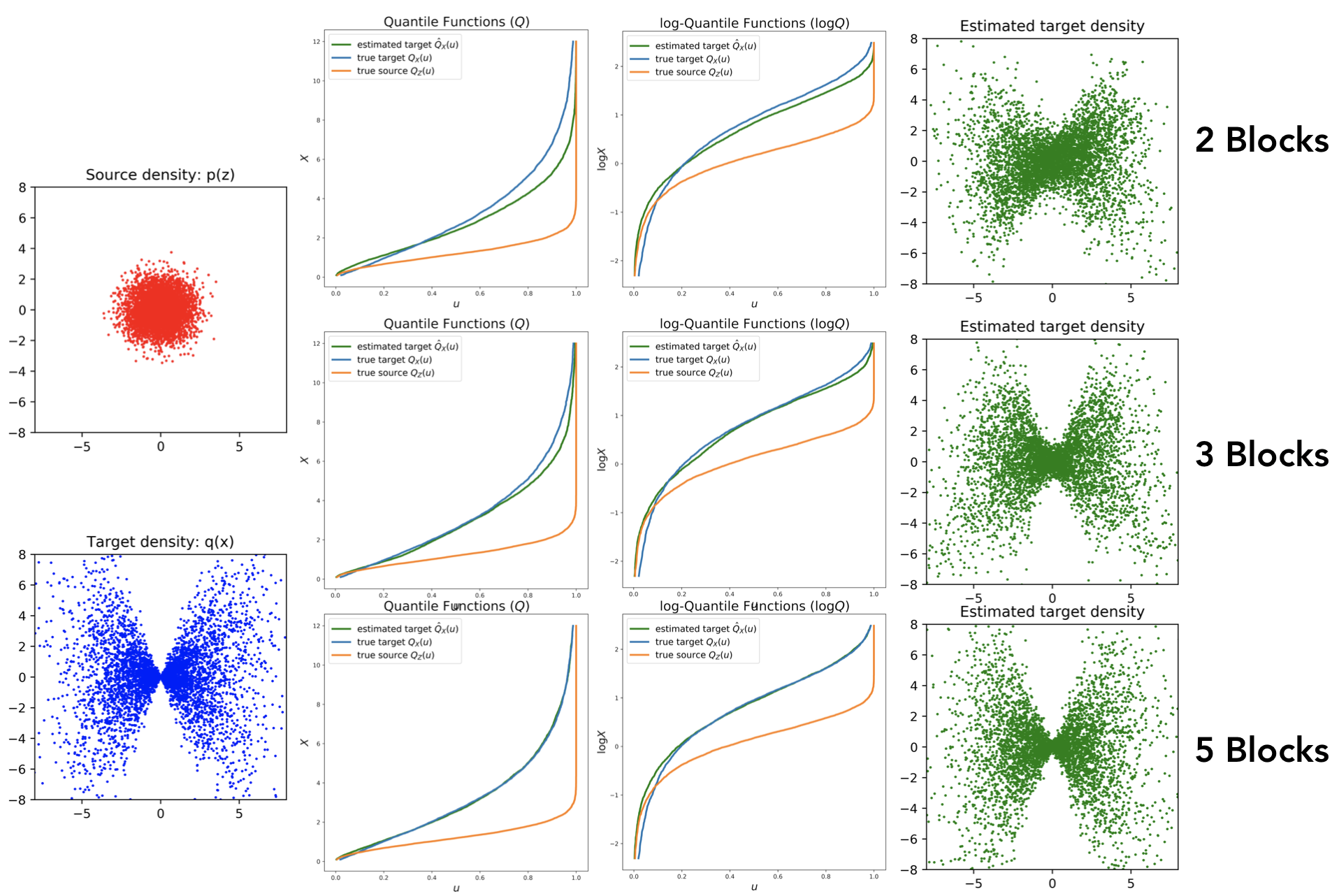}
    \caption{Results for tail-adaptive Real-NVP on Neal's funnel distribution as target density. Figure organization is same as \Cref{fig:sig-nvp}.}
    \label{fig:neal}
    \vspace{-1em}
\end{figure} 

%% file: section/con.tex
\section{Conclusion}
\label{sec:con}
\vspace{-0.5em}
We studied the ability of popular flow models to capture tail-properties of a target density by studying the corresponding increasing triangular map approximated by these flow methods acting on a tractable source density with known fixed tails. We showed that any Lipschitz-continuous transport map cannot push a source density to a heavier target density, implying that affine flow models like Real-NVP, NICE, MAF, Glow \etc cannot capture heavier tails than the source density. We then propose tail-adaptive flows (TAFs) where the tails of the source density can be adapted during training. TAFs are appealing because they can be substituted easily in existing flow architectures and optimized using automatic differentiation.  Further, their ability to adapt tails of the source density allows affine TAFs to learn heavier tailed distributions. In future work, we will be interesting to explore the applications of TAFs for extreme value theory and financial risk analysis. 

\vspace{-0.5em}
\section*{Acknowledgement}
\vspace{-0.5em}
We thank Andy Keller, Didrik Nielsen, Jorn Peters, and Patrick Forre for discussions and feedback. We also thank the anonymous reviewers for their valuable feedback. We would also like to acknowledge NSERC, the Canada CIFAR AI Chairs Program, and MITACS Accelerate for financial support. We thank NVIDIA Corporation (the data science grant) for donating two Titan V GPUs that enabled in part the computation in this work. PJ was additionally supported by a Borealis AI fellowship and Huawei Graduate fellowship.